\documentclass[11pt]{amsart}

\usepackage{amssymb}
\usepackage{amsmath}
\usepackage{amsthm}
\usepackage{dsfont}

\theoremstyle{plain}
\newtheorem{theorem}{Theorem}[section]
\newtheorem{prop}[theorem]{Proposition}
\newtheorem{lemma}[theorem]{Lemma}
\newtheorem{remark}[theorem]{Remark}

\title{Root number of the Jacobian of $y^2=x^p+a$}
\author{Matthew Bisatt}
\address{Fry Building, University of Bristol, Bristol, BS8 1UG, UK}
\email{matthew.bisatt@bristol.ac.uk}

\date{\today}

\begin{document}
\global\long\def\Q{\mathbb{Q}}
\global\long\def\Z{\mathbb{Z}}
\global\long\def\qp{\Q_p}
\global\long\def\fp{\mathbb{F}_p}
\global\long\def\Ind{\operatorname{Ind}}
\global\long\def\Gal{\operatorname{Gal}}
\global\long\def\Tr{\operatorname{Tr}}
\global\long\def\Frob{\operatorname{Frob}}
\global\long\def\sign{\operatorname{sign}}
\global\long\def\Norm{\operatorname{Norm}}
\global\long\def\Jac{\operatorname{Jac}}
\global\long\def\OK{\mathcal{O}_K}

\begin{abstract}
Let $C/\Q$ be a hyperelliptic curve with an affine model of the form $y^2=x^p+a$. We explicitly determine the root number of the Jacobian of $C$, with particular focus on the local root number at $p$ where $C$ has wild ramification.
\end{abstract}

\maketitle

\section{Global case}

Let $A/\Q$ be an abelian variety. One of the consequences of the famed Birch and Swinnerton-Dyer conjecture is that the parity of its rank should be controlled by the expected sign in its functional equation. The expected sign is known as the (global) root number of $A$ and denoted by $W(A/\Q)$; the parity conjecture states that $W(A/\Q)=(-1)^{\operatorname{rank} A(\Q)}$.

Root numbers of elliptic curves are completely understood, but the same cannot be said in higher dimensions. Whilst root numbers only conjecturally control the rank, current computer algebra packages are often unable to compute either the $L$-function or rank explicitly in high dimension; when these computations are coupled with the root number however, one is more likely to be able to predict the rank exactly. For example, a root number calculation may predict that a Jacobian of a curve has rank $0$, and thus make the curve more amenable to Chabauty-type arguments since one generally needs to know explicit generators of a finite index subgroup to apply them.

The author has previously described the root number when $A$ has everywhere tame reduction \cite{Bis19} with an explicit description when $A$ arises as the Jacobian of a hyperelliptic curve. The purpose of this note is to compute the root number of Jacobian of a hyperelliptic curve of the form $y^2=x^p+a$ when $p$ is an odd prime. We concentrate on the case when $a$ is not a $p$-th power in $\qp$ although this restriction is not necessary (see Remark \ref{redcase}). The main theoretical contributions of this paper are that we extend a result of Kobayashi's to deal with the wild ramification at $p$ and also give an approach to treat the case $\ell=2$ where one cannot use the theory of cluster pictures of hyperelliptic curves to determine the Galois representation.

We will not list the various properties of root numbers that we use here, but instead refer the reader to \cite[\S 2]{Bis19} for an overview or \cite[\S 3]{Tat79} for more details.

For a prime $p$, we write $v_p$ for the $p$-adic valuation, $\left(\frac{\ast}{p}\right)$ for the Legendre symbol and $(-,-)_{\qp}$ for the Hilbert symbol at $p$.

\begin{theorem}
Fix $0 \neq a \in \Z$ and an odd prime $p$. Let $\mathcal{C}_a/\Q:y^2=x^p+a$ be the hyperelliptic curve with Jacobian $J_a$. Assume that $a$ is not a $p$-th power in $\qp$ and let $a'=2^{-v_2(a)}a$.
\begin{enumerate}
\item Let $W_2= \begin{cases}
\left(\frac{2}{p}\right) & \!\!\!\! \text{ if } v_2(a) \text { is even, $p \nmid (v_2(a)+2)$, and } a' \equiv 1 \!\! \mod{4}, \\
\left(\frac{-1}{p}\right) & \!\!\!\! \text{ if } a' \equiv 3 \mod{4}, \\
1 & \!\!\!\! \text{ else.}
\end{cases}$
\item Assume there exists a model of $\mathcal{C}_a$ of the form $y^2=g(x)$, where $p \nmid v_p(g(0))$. Let $a_g=g(0)$ and let $\Delta_g$ be the discriminant of $g$. Define $$W_p=- \left(\frac{-2}{p}\right) (a_g(p-1)v_p(a_g),\Delta_g)_{\qp} \left( \frac{-1}{p}\right)^{\frac{1+v_p(\Delta_g)}{2}}.$$
\end{enumerate}

Write
\begin{eqnarray*}
	S_1 &=& \{\text{primes } \ell \mid a : \,\, \ell \not\in \{2,p\}, \quad v_{\ell}(a) \text{ odd}  \}, \\
	S_2 &=& \{\text{primes } \ell \mid a : \,\, \ell \not\in \{2,p\}, \quad p \nmid v_{\ell}(a), \,\, v_{\ell}(a) \text{ even}  \}.
\end{eqnarray*}

Then the global root number of $J_a$ is
$$W(J_a/\Q)=W_2W_p \left(\frac{-1}{p}\right) \prod\limits_{\ell \in S_1} \left(\frac{-1}{\ell}\right)^{\frac{p-1}{2}} \prod\limits_{\ell \in S_2} \left(\frac{\ell}{p}\right).$$
\end{theorem}

\begin{remark}
In practice, it appears that taking $g(x)=f(x+r)$ for some $r \equiv -a \mod{p}$ suffices to obtain a suitable model. For $p=3$, it is proved in \cite[Proposition 4.1]{Kob02} that a suitable change of model exists but we do not generalise this here.
\end{remark}

\begin{remark}
\label{redcase}
The assumption that $a$ is not a $p$-th power in $\qp$ is only used to compute the local root number $W_p=W(J_a/\qp)$. If $a$ is a $p$-th power in $\qp$ then one can still compute that the inertia representation at $p$ is tamely ramified and isomorphic to $\gamma \otimes \mathbb{C}[C_{p-1}]$, where $\gamma$ has order $2(p-1)$, and then deduce the root number by \cite[\S 11, \S 13]{HEC20}. We choose to omit the result here in the interest of brevity.
\end{remark}


\begin{proof}
Note that the discriminant of $\mathcal{C}_a$ is $(-1)^{\frac{p-1}{2}}2^{2(p-1)}p^pa^{p-1}$ and hence $$W(J_a/\Q)=W(J_a/\mathbb{R})W(J_a/\Q_2)W(J_a/\Q_p)\prod\limits_{\ell \mid a, \, \ell \neq 2,p} W(J_a/\Q_{\ell}).$$ Moreover, one has that $W(J_a/\mathbb{R})=\left(\frac{-1}{p}\right)$ (see for example \cite[Lemma 2.1]{Sab07}).

Now suppose that $\ell \neq p$ is an odd prime. We will use the theory of clusters to compute the corresponding inertia representation and refer the reader to the user's guide \cite{HEC20} for the relevant definitions and theorem statements. The set of roots is $\mathcal{R}=\{\zeta_{p}^j \sqrt[p]{-a} \, \mid \, 0\leq j\leq p-1 \}$, for $\zeta_p$ a primitive $p$-th root of unity; this is the only proper cluster of $\mathcal{C}_a$ and has depth $\frac{v_{\ell}(a)}{p}$. Moreover, inertia acts on $\mathcal{R}$ through a $C_p$-quotient and the action is trivial if and only if $p \mid v_{\ell}(a)$.

If $p \nmid v_{\ell}(a)$, then $H^1_{\acute{e}t}(J_a/\Q_{\ell},\overline{\Q}_2)$ decomposes as an inertia representation into $p-1$ characters of order either $2p$ if $v_{\ell}(a)$ is odd and order $p$ otherwise. The corresponding local root number is then either $\left(\frac{-1}{\ell}\right)^{\frac{p-1}{2}}$ or $\left(\frac{\ell}{p}\right)$ respectively by \cite[Theorem 13.4]{HEC20}. On the other hand, if $p \mid v_{\ell}(a)$, then $H^1_{\acute{e}t}(J_a/\Q_{\ell},\overline{\Q}_2)$ decomposes into $p-1$ characters of order either $2$ if $v_{\ell}(a)$ is odd and order $1$ otherwise; in this case the local root number is then $\left(\frac{-1}{\ell}\right)^{\frac{p-1}{2}}$ or $1$ respectively.

Next we compute the local root number at $2$, for which we can no longer apply the cluster machinery. First suppose that $v_2(a)=2k$ is even and $a'=a2^{-2k} \equiv 1 \mod{4}$. Using the substitution $(x,y) \mapsto (x,y+2^k)$, we change our model for $\mathcal{C}_a/\Q_2$ to $y^2+2^{k+1}y=x^p+4^k(a'-1)$. Let $\pi=\sqrt[p]{2}$ and let $L=\Q_2(\pi^{2(k+1)})$. Consider the base change $\mathcal{C}_L$ of $\mathcal{C}_a$ to $L$ with the substitution $x=\pi^{2(k+1)}X, y=2^{k+1}Y$; this gives us a model for $\mathcal{C}_L$ as $Y^2+Y=X^{11}+\frac{a'-1}{4}$ which has good reduction over $L$. Now $L=\Q_2$ if and only if $p \mid (k+1)$; in this case the root number is $1$. If $p \nmid (k+1)$, then $L=\Q_2(\pi)$ and using the lift-act-reduce procedure of \cite[Thereom 1.5]{DDM20} (this example is analogous to Example 1.9 of \emph{loc. cit.}) one can moreover see that the inertia representation is $\mathbb{C}[C_p] \ominus \mathds{1}$ and hence the local root number is $W(J_a/\Q_2)=\left(\frac{2}{p}\right)$.

The remaining cases at $2$ are quadratic twists of the one above so we will make use of the well-known formula $$W(A/\Q_{\ell})W(A^D/\Q_{\ell})=W(A/\Q_{\ell}(\sqrt{D}))\chi_D(-1)^{\dim A},$$ where $A/\Q_{\ell}$ is an abelian variety, $A^D$ its quadratic twist by $D$, and $\chi_D$ is the quadratic character of $\Gal(\Q_{\ell}(\sqrt{D})/\Q_{\ell})$.

Note that if $\tilde{a} \equiv 1 \bmod{4}$, then $\mathcal{C}_a$ is a quadratic twist of a curve of the form $\mathcal{C}_{2^{2k}\tilde{a}}$ for some $D \in \{-1,2,-2\}$; if $a'=2^{-v_2(a)}a$, then the required quadratic twist is given by the squarefree part of $(-1)^{\frac{a'-1}{2}} 2^{v_2(a)}$.

Since $\mathcal{C}_{2^{2k}\tilde{a}}$ obtains good reduction over $L$ which has ramification degree dividing $p$ and $\Q_2(\sqrt{D})/\Q_2$ is ramified of degree coprime to $p$, one has that $W(J_{2^{2k}\tilde{a}}/\Q_2)=W(J_{2^{2k}\tilde{a}}/\Q_2(\sqrt{D}))$ (note it does not matter whether $L=\Q_2$ here). Hence $W(\Jac(\mathcal{C}_{2^{2k}\tilde{a}}^D)/\Q_2)=\chi_D(-1)^{\frac{p-1}{2}}$. One can compute via class field theory that $\chi_{-1}(-1)=\chi_{-2}(-1)=-1$ and $\chi_{2}(-1)=1$, and hence the local root number $W_2=W(J_a/\Q_2)$ follows.

Now let $\ell=p$ and consider the model change to $y^2=g(x)$. Note that the two torsion field of $J_a$ is independent of the choice of model and equal to the splitting fields of $f$ and $g$. Since $a$ is not a $p$-th power in $\qp$, $g$ satisfies the hypotheses of Theorem \ref{localRN} and the local root number $W_p=W(J_a/\qp)$ follows.
\end{proof}

\begin{remark}
When $\ell=2$, one could also obtain the inertia representation in the tame case via \cite[\S 1.3]{Dok18}.
\end{remark}

\section{Local case at $p$}

We now compute the local root number at $p$. Our method closely follows that of \cite{Kob02}, adapting for higher genus. We first collate the bulk of the notation that we will use throughout this section for ease of reading. \\

\noindent \textbf{Notation.}
For a finite extension $F/\qp$, we denote

\begin{tabular}{cl}
$\mathcal{O}_F$ & the valuation ring of $F$, \\
$v_F$ & the normalised valuation of $F$ so that $v_F(F^{\times})=\Z$, \\
$\mathbb{F}_F$ & the residue field of $F$, \\
$q_F$ & $=|\mathbb{F}_F|$ the cardinality of the residue field of $F$, \\
$F^{un}$ & the maximal unramified extension of $F$, \\
$\mathcal{W}(F'/F)$ & the Weil group of an extension $F'/F$, \\
$\Tr_{F'/F}$ & the trace map for a finite extension $F'/F$.
\end{tabular}

For an irreducible squarefree polynomial $f \in K[x]$ of degree $p$, we also define:

\begin{tabular}{cl}
$\mathcal{C}_f$ & the hyperelliptic curve $y^2=f(x)$, \\
$\rho_J$ & the $\ell$-adic representation of $\Jac(\mathcal{C}_f)$ for $\ell \neq p$, \\
$\alpha$ & a root of $f$, \\
$\Delta_f$ & the discriminant of $f$, \\
$M$	& the splitting field of $f$, \\
$\sigma$ & a fixed element of order $p$ in the Galois group $\Gal(M/K)$, \\
$N$ & $=M(\sqrt{\alpha - \sigma\alpha})$.
\end{tabular}

For nonzero complex numbers $a,b$, we write $a \approx b$ if $ab^{-1} \in \mathbb{C}$ is in the multiplicative group generated by the positive real numbers and the $p$-power roots of unity. We are now in a position to state our theorem.

\begin{theorem}
\label{localRN}
Let $p$ be an odd prime, $K/\qp$ a finite extension and let $f \in K[x]$ be a squarefree polynomial of degree $p$. Suppose that:
\begin{enumerate}
\item $f$ is irreducible;
\item The valuation of the discriminant of $f$, $v_K(\Delta_f)$, is coprime to $p-1$;
\item The splitting field $M$ of $f$ has degree $p(p-1)$;
\item $p \nmid v_K(a_0)$ where $a_0=f(0)$.
\end{enumerate}
Then, for the hyperelliptic curve $\mathcal{C}_f:y^2=f(x)$, the local root number is $$W(\Jac(\mathcal{C}_f)/K)=- \left(\frac{-2}{q_K}\right) (a_0(p-1)v_K(a_0),\Delta_f)_{K} \left( \frac{-1}{q_K}\right)^{\frac{1+v_K(\Delta_f)}{2}}.$$
\end{theorem}

We will implicitly assume these hypotheses from now on in order to prove this theorem through a series of lemmas.

\newpage

\begin{prop}
\begin{enumerate}
\item[]
\item Let $\sigma$ be an element of order $p$ of $\Gal(M/K) \cong C_p \rtimes C_{p-1}$ and let $N=M(\sqrt{\alpha-\sigma\alpha})$. Then $\rho_J$ is a faithful representation of the Weil group $\mathcal{W}(N^{un}/K)$.
\item Let $H \subset M$ be such that $[H:K]=p-1$. Then $\rho_J=\Ind_{H/K} \chi$ is the induction of a character $\chi$ from $\mathcal{W}(N^{un}/H)$;
\item The conductor exponent of $\chi$, $a(\chi)$, is even.
\end{enumerate}
\end{prop}

\begin{proof}
(1)-(2) follow from Proposition 2.2 and \S 5.1 of \cite{Cop20}.

(3)  Note that $a(\rho_J)=v_K(\Delta_{K(\alpha)/K})$, where $\Delta_{K(\alpha)/K}$ is the relative discriminant of the valuation ring of $K(\alpha)$, by \cite[\S 6]{Cop20}; we claim this is odd. Since the discriminant of the lattice $\OK[\alpha]$ is equal to $\Delta_f$ and has square index inside the valuation ring, $v_K(\Delta_{K(\alpha)/K}) \equiv v_K(\Delta_f) \mod{2}$ from which the claim follows. By the conductor exponent of induced representations formula \cite[p.101]{Ser79}, $a(\chi)$ is even.
\end{proof}

We now fix an additive character $\psi$ of $K$ with conductor $-1$ such that $\psi(x)=\chi(\sigma)^{-\Tr_{\mathbb{F}_K/\fp} (\bar x)}$ whenever $x \in \OK$, where $\Tr_{\mathbb{F}_K/\fp}$ is the trace map from $\mathbb{F}_K$ to $\fp$ and $\bar x$ is the image of $x$ in $\mathbb{F}_K$. We further let $\psi_H$ be the composition of $\psi$ with the trace map $\Tr_{H/K}$, also of conductor $-1$. We also choose an arithmetic Frobenius element $\Frob$ of $\mathcal{W}(N^{un}/N) \subset \mathcal{W}(N^{un}/K)$ and extend $\sigma$ to $\mathcal{W}(N^{un}/K)$ by imposing that it acts trivially on $K^{un}$ and $\sqrt{\alpha-\sigma\alpha}$.

\begin{lemma}
\label{ind}
Let $\sign_{H/K}$ be the nontrivial quadratic character of $\Gal(H/K)$. Then $W(\rho_J,\psi)=\left(\dfrac{-2}{q_K}\right)W(\operatorname{sign}_{H/K},\psi_H)W(\chi,\psi_H)$.
\end{lemma}

\begin{proof}
For $F \subset N$, let $\mathds 1_F$ denote the trivial character of $\mathcal{W}(N^{un}/F)$. Since root numbers are inductive in degree $0$, one has $W(\Ind_{H/K} \chi, \psi)=W(\Ind_{H/K} \mathds 1_H,\psi)W(\chi, \psi_H)$.

We now compute $W(\Ind_{H/K} \mathds 1_H, \psi)$. Let $\mu$ be a faithful character of $\mathcal{W}(N^{un}/K)/\mathcal{W}(N^{un}/H) \cong \Gal(H/K) \cong C_{p-1}$. Then $$\Ind_{H/K} \mathds 1_H = \mathds{1}_{K} \oplus \, \sign_{H/K} \, \oplus \bigoplus\limits_{1 \leqslant i < \frac{p-1}{2}} (\mu^j \oplus \mu^{-j}),$$ and hence $W(\Ind_{H/K} \mathds 1, \psi) = W(\mathds{1}_{K},\psi)W(\sign_{H/K}, \psi)\prod\limits_{1 \leqslant j < \frac{p-1}{2}} \mu^j(-1)$. Now $W(\mathds{1}_{K},\psi)=1$ and $\mu^j(-1) = 1 \Leftrightarrow q \equiv 1 \bmod{2\frac{p-1}{\gcd(j, p-1)}}$ by \cite[Lemma 3.8]{Bis19}. If $q_K$ is an even power of $p$, then this is always satisfied since $2(p-1) \mid (q_K-1)$. On the other hand, if $q_K$ is an odd power of $p$, then this congruence is satisfied if and only if $j$ is even so it suffices to count the number of odd integers in the range $1 \leqslant j < \frac{p-1}{2}$; the parity of this precisely given by the Legendre symbol $\left(\frac{-2}{p}\right)$.
\end{proof}

\begin{remark}
	One of our running assumptions is that the splitting field of $f$ has degree $p(p-1)$ over $K$; the preceding lemma is the only place we made use of this fact since this implies that $H/K$ is a cyclic Galois extension.
\end{remark}

We briefly generalise Propositions 5.7 and 5.13 of \cite{Kob02} now.

\begin{lemma}
\label{endeq}
Let $\overline{\mathcal{C}}/\mathbb{F}_q:y^2=x^p-x$ and let $\overline\Frob$ be the Frobenius endomorphism of $\overline{\mathcal{C}}$ given by $(x,y)\mapsto (x^{q},y^{q})$ and let $\overline\rho$ be the endomorphism $(x,y) \mapsto (x+1,y)$. By Albanese functoriality, these descend to endomorphisms of the Jacobian $\overline{J}$ of $\overline{\mathcal{C}}$ which we again denote by $\overline\Frob$ and $\overline\rho$. Then in the endomorphism ring of $\overline{J}$, one has $\overline\Frob = - \sum\limits_{u \in \mathbb{F}_q^{\times}} \left(\frac{u}{q}\right)\overline\rho^{\Tr_{\mathbb{F}_q/\fp}(u)}.$
\end{lemma}

\begin{proof}
By the Hasse--Davenport theorem, it suffices to prove this when $q=p$. Let $\infty$ denote the point at infinity of $\overline{\mathcal{C}}$ and embed $\iota: \overline{\mathcal{C}} \hookrightarrow \overline{J}$ via $P \mapsto [(P) - (\infty)]$. Let $P=(x,y)$; it suffices to prove this equality for $\iota(P)$ since $\overline{J}$ is generated by the image of $\iota$. Using the hyperelliptic involution, we have
$$\left(\overline\Frob + \sum_{u \in \fp^{\times}} \left(\frac{u}{p} \right)\overline\rho^u \right)(\iota(P))=(x^p,y^p)+ \sum_{u \in \fp^{\times}} (x+u,\left(\frac{u}{p}\right)y) - p(\infty).$$
One can check that this is the divisor of $Y-y(X-x)^{\frac{p-1}{2}}$ and hence is trivial on $\overline{J}$, which proves the equality.
\end{proof}

\begin{prop}
\label{opeq}
	As operators of $H^1_{\acute{e}t}(\mathcal{C}_f/\overline{K},\Q_{\ell})$, we have an equality $\Frob = -\sum\limits_{u \in \mathbb{F}_K^{\times}} \left(\frac{u}{q_K}\right) \sigma^{-\Tr_{\mathbb{F}_K/\fp}(u)}$.
\end{prop}

\begin{proof}
	Note that over $N$, the model o $\mathcal{C}_f$ determined by $(x,y) \mapsto ((\alpha-\sigma\alpha)x+\alpha, \sqrt{\alpha-\sigma\alpha}^py)$ reduces to $y^2=x^p-x$ (see \cite[Lemma 4.1]{Cop20}). Moreover, one can show via the lift-act-reduce method of \cite[Theorem 1.5]{DDM20}, that $\Frob$ and $\sigma^{-1}$ act on $H^1_{\acute{e}t}(\mathcal{C}_f/\overline{\mathbb{F}}_K,\Q_{\ell})$ as $(x,y) \rightarrow (x^{q_K},y^{q_K})$ and $(x,y) \rightarrow (x+1,y)$ respectively (equivalently use the function field approach of \cite[Proposition 5.12]{Kob02}). The Galois module isomorphism of \cite[Corollary 1.6]{DDM20} then transfers the operator equality of Lemma \ref{endeq} to $H^1_{\acute{e}t}(\mathcal{C}_f/\overline{K},\Q_{\ell})$.
\end{proof}

\begin{lemma}
\label{chi}
Let $G$ be the Gauss sum $\sum\limits_{u \in \mathbb{F}_K^{\times}} \left(\frac{u}{q_K}\right) \chi(\sigma)^{-\Tr_{\mathbb{F}_K/\fp} (u)}$. Then $$W(\chi,\psi_H) \approx (a_0v_M(\alpha),\Delta_f)_{K} (-G)^{v_K(\Delta_f)}.$$
\end{lemma}

\begin{proof}
First note that $W(\chi,\psi_H) \approx \chi \left(\frac{-\Norm_{M/H}(\alpha-\sigma\alpha)}{a_0v_M(\alpha)^p} \right)$ by the proofs of Propositions 4.2 and 5.6 of \cite{Kob02}. The remainder is largely \cite[Proposition 5.8]{Kob02}.

Note $\chi \!\! \mid_{K^{\times}} = \sign_{H/K} \det \rho_J$. Since $\det \rho_J$ is an integer power of the cyclotomic character and $\sign_{H/K}$ is the Hilbert symbol $(-,\Delta_f)_{K}$, we have $\chi(a_0v_M(\alpha)^p) \approx (a_0v_M(\alpha)^p,\Delta_f)_{K}=(a_0v_M(\alpha),\Delta_f)_{K}$.

Let $\beta=-\Norm_{M/H}(\alpha -\sigma\alpha)$. We now claim that $v_H(\beta)=v_K(\Delta_f)$. Since $\mathcal{C}_f$ has potentially good reduction, $v_M(\sigma^i\alpha -\sigma^j\alpha)=v_M(\alpha-\sigma\alpha)$ whenever $\sigma^i \neq \sigma^j$ by \cite[Theorem 5.3]{HEC20}. Hence $v_M(\Delta_f)=p(p-1)v_M(\alpha-\sigma\alpha)=p(p-1)v_H(\Norm_{M/H}(\alpha -\sigma\alpha))$. The claim follows since $M/K$ is totally ramified of degree $p(p-1)$.

Observe that $\beta=\Norm_{N/H}(\sqrt{\alpha-\sigma\alpha})$ is a norm from $N$ and so by Proposition \ref{opeq} with the proof of \cite[Proposition 5.8]{Kob02}, $\chi(\beta)=\chi(\Frob^{v_H(\beta)})=(-G)^{v_K(\Delta_f)}.$
\end{proof}

Finally, we can amalgamate these results to prove Theorem \ref{localRN}.

\begin{proof}[Proof of Theorem \ref{localRN}]
By direct computation, $W(\sign_{H/K},\psi) \approx G$. Moreover, $W(\rho_J,\psi) \approx -\left(\frac{-2}{q_K}\right) (a_0v_M(\alpha),\Delta_f)_{K} (-G)^{1+v_K(\Delta_f)}$ by Lemmas \ref{ind} and \ref{chi}, and $v_M(\alpha)=(p-1)v_K(a_0)$. Since $v_K(\Delta_f)$ is odd and $G^2$ is real with sign $\left(\frac{-1}{q_K}\right)$ we get the result up to $\approx$ a priori; the equality follows since both sides are $\pm 1$.
\end{proof}

\textbf{Acknowledgements.} I wish to thank Tim Dokchitser for his suggestion when $\ell=2$ and Nirvana Coppola for answering questions on her paper.

\bibliographystyle{alpha}
\bibliography{coprime}

\end{document}